\documentclass[12pt]{article}

\usepackage{mathrsfs}

\usepackage{amsmath,amsthm,amssymb}
\usepackage{graphicx}
\usepackage{vector}

\font\tencmmib=cmmib10 \skewchar\tencmmib '60
\newfam\cmmibfam
\textfont\cmmibfam=\tencmmib

\def\lessim{\ \lower4pt\hbox{$
\buildrel{\displaystyle <}\over\sim$}\ }
\def\gessim{\ \lower4pt\hbox{$\buildrel{\displaystyle >}
\over\sim$}\ }

\newcommand{\e}{\mathbb{E}}

\newcommand{\hp}{\hat{\mathcal{P}}}
\newcommand{\hg}{\hat{\Gamma}}

\newtheorem{lemma}{\bf Lemma}

\newtheorem{theorem}{\bf Theorem}
\newtheorem{corollary}{\bf Corollary}
\newtheorem{remark}{\bf Remark}

\newtheorem{proposition}{\bf Proposition}

\newenvironment{Proof of lemma}{\noindent{\bf Proof of Lemma}}{\hfill$\Box$\newline}
\newenvironment{Proof of theorem}{\noindent{\bf Proof of Theorem}}{\hfill{\footnotesize${\square}$}\newline}
\newenvironment{Proof of theorems}{\noindent{\bf Proof of Theorems}}{\hfill$\Box$\newline}
\newenvironment{Proof of proposition}{\noindent{\bf Proof of Proposition}}{\hfill$\Box$\newline}
\newenvironment{Proof of propositions}{\noindent{\bf Proof of Propositions}}{\hfill$\Box$\newline}
\newenvironment{Proof of exercise}{\noindent{\it Proof of Exercise:}}{\hfill$\Box$}

\begin{document}

\title{The Legendre structure of the Parisi formula}
\author{Antonio Auffinger  \thanks{auffing@math.northwestern.edu } \\ \small{Northwestern University}\and Wei-Kuo Chen \thanks{wkchen@umn.edu} \\ \small{University of Minnesota} }
\maketitle

\begin{abstract}
	We show that the Parisi formula of the mixed $p$-spin model is a concave function of the squared inverse temperature. This allows us to derive a new expression for the Parisi formula that involves the inverse temperature and the Parisi measure as Legendre conjugate variables.
\end{abstract}


\section{Introduction and Main results}

The Sherrington-Kirkpatrick (SK) model is one of the most fundamental models of spin glasses. Based on the so-called replica method, an attempt to fully describe the behavior of the system was proposed in the astounding work of Parisi \cite{Pa79,Pa80}. In Parisi's theory, the thermodynamic limit of the free energy can be computed through a minimization problem, known as the Parisi formula. Over the past decades, in a series of groundbreaking works by Guerra \cite{G03}, Talagrand \cite{Tal11I,Tal11II}, and Panchenko \cite{P13}, the Parisi formula was proved in a general class of mean-field spin glasses, the mixed $p$-spin model. Its Hamiltonian is defined as
$$
-H_N(\sigma)=\sum_{p=2}^\infty\frac{c_p}{N^{(p-1)/2}}\sum_{1\leq i_1,\ldots,i_p\leq N}g_{i_1,\ldots,i_p}\sigma_{i_1}\cdots\sigma_{i_p}
$$
for $\sigma = (\sigma_{1}, \ldots, \sigma_{N})\in \Sigma_N:=\{-1,+1\}^N$, where $g_{i_1,\ldots,i_p}$'s are i.i.d. standard Gaussian for all $1\leq i_1,\ldots,i_p\leq N$ and $p\geq 2$. Here the real sequence $(c_p)$ is assumed to decay fast enough, e.g., $\sum_{p=2}^\infty2^pc_p^2<\infty$ so that the covariance of $H_N$ can be computed as
$$
\e H_N(\sigma^1)H_N(\sigma^2)=N\xi(R_{1,2}),
$$
where $
\xi(t):=\sum_{p=2}^\infty c_p^2t^{p}
$
and 
$
R_{1,2}:=N^{-1}\sum_{i=1}^N\sigma_i^1\sigma_i^2
$
is the overlap between two spin configurations $\sigma^1$ and $\sigma^2.$ In particular, the SK model is $\xi(t)=t^2/2.$

The Parisi formula is described as follows. Let $\mathcal{M}$ be the space of all distribution functions on $[0,1]$ endowed with the Lebesgue $L^1$-norm on $[0,1]$. For each $\alpha\in\mathcal{M}$ and $\beta\geq 0$, let $\Phi_{\alpha,\beta}$ be the solution to the Parisi PDE, 
\begin{align}
\label{pde}
\partial_t\Phi_{\alpha,\beta}(t,x)=-\frac{\beta^2\xi''(t)}{2}\bigl(\partial_{xx}\Phi_{\alpha,\beta}(t,x)+\alpha(t)(\Phi_{\alpha,\beta}(t,x))^2\bigr)
\end{align}
for $(t,x)\in[0,1]\times\mathbb{R}$ with terminal condition $\Phi_{\alpha,\beta}(1,x)=\log \cosh x.$ We remark that this equation is solvable in the classical sense by performing the Hopf-Cole transformation when $\alpha$ is a step function, while in the general case, the solution $\Phi_{\alpha,\beta}$ should be understood in the weak sense, see \cite{JT15}.
Define the Parisi functional by
\begin{align*}
\mathcal{P}(\alpha,\beta)=\log 2+\Phi_{\alpha,\beta}(0,0)-\frac{\beta^2}{2}\int_0^1 \alpha(s)s\xi''(s)ds
\end{align*}
and the Parisi variational problem by 
\begin{align}
\label{eq:variational}
\mathcal{P}(\beta)=\min_{\alpha\in \mathcal{M}}\mathcal{P}(\alpha,\beta).
\end{align}
The famous Parisi formula says that the thermodynamic limit of the free energy can be computed through
\begin{align}\label{eq:ParisiEq23}
\lim_{N\rightarrow\infty}\frac{1}{N}\e\log \sum_{\sigma}\exp(-\beta H_N(\sigma))=\mathcal{P}(\beta).
\end{align} 
Following Guerra's discovery of replica symmetry breaking bound \cite{G03}, the first rigorous proof of this formula was given by Talagrand \cite{Tal06} in the setting of the mixed even $p$-spin model. The proof was later extended to the mixed $p$-spin model with odd $p$-spin interactions by Panchenko \cite{P12}. Recently, the authors proved in \cite{AChen14} that $\Phi_{\alpha,\beta}(0,0)$ defines a strictly convex functional in $\alpha$ with respect to the Lebesgue $L^1$-norm on $[0,1]$. In particular, this result established uniqueness of the minimizer in the Parisi variational problem \eqref{eq:variational}. Throughout this paper, we shall call such minimizer the Parisi measure and denote it by $\alpha_{P,\beta}$. 
For the qualitative properties of the Parisi measure, we refer the readers to \cite{AChen}.

Despite the facts that the Parisi formula was formulated 35 years ago and Talagrand's proof has appeared for more than a decade, Parisi's solution remains puzzling and counter-intuitive in many aspects. The well-known entropic principle in statistical mechanics suggests that the thermodynamic limit of the free energy should be written as a maximization problem in terms of the entropy and the internal energy over the generic thermodynamic states. Similarly, in large deviation theory, the Laplace-Varadhan formula advocates for a maximization involving an energy functional and a rate function. 
Nonetheless, these two methods fail short in providing a representation for the limiting free energy. As far as we know, a LDP approach was only successful in the study of the Random Energy Model (REM) and the Generalized Random Energy Model (GREM) as well as their variants in the papers of Bolthausen and Kistler \cite{EN08,EN09}.

More importantly, these theories indicate that both pairs (entropy and free energy in thermodynamics, energy functional and rate function in large deviation limits) should be related through a Legendre transformation. The aim of this paper is to establish such Legendre structure within the framework of Parisi's solution; we derive a new representation for $\mathcal P(\beta)$  that exhibits a Legendre duality, where the square of the inverse temperature and the functional order parameter are conjugate variables. This work is motivated by a recent ingenious talk given by Guerra at the institute of Henri Poincar\'e. He first conjectured that the thermodynamic limit of the free energy is concave in the squared inverse temperature. Second, if such concavity was indeed valid, the Parisi formula could be written as a new minimization \eqref{thm1:eq1} in Legendre form. 
In this paper, we give an affirmative answer to both Guerra's conjectures. Our first main result establishes the concavity of the reparametrized Parisi PDE solution in the squared inverse temperature. 

\begin{theorem}\label{thm2}
	For each $\alpha\in\mathcal{M},$ $\Phi_{\alpha,\sqrt{\gamma}}(0,0)$ is concave in $\gamma.$
\end{theorem}

It is well understood that the original Parisi PDE solution $\Phi_{\alpha,\beta}$ is convex in $\beta$. Theorem \ref{thm2} goes in the completely opposite direction when the inverse temperature is reparametrized. Consider the Parisi variational problem \eqref{eq:variational} associated to the squared inverse temperature $\gamma=\beta^2$, that is, 
\begin{align*}
\hat{\mathcal{P}}(\gamma)&:=\mathcal{P}(\beta).
\end{align*}
Likewise, we set the reparametrized Parisi functional as
$$
\hp(\alpha,\gamma)=\mathcal{P}(\alpha,\beta)=\log 2+\Phi_{\alpha,\sqrt{\gamma}}(0,0)-\frac{\gamma}{2}\int_0^1 \alpha(s)s\xi''(s)ds.
$$
From Theorem \ref{thm2}, one sees that $\hp(\alpha,\gamma)$ is concave in $\gamma$ and consequently induces the concavity of $\hp(\gamma)$. Set the Legendre transform of $\hat{\mathcal{P}}(\alpha,\gamma)$ as
\begin{align*}
\hg(\alpha)&=\sup_{\gamma\in (0,\infty)}\Bigl(\hp(\alpha,\gamma)-\frac{\gamma}{2}\int_0^1\alpha(s)\xi'(s)ds\Bigr),\,\,\forall \alpha\in\mathcal{M}.
\end{align*}
Note that $\hg(\alpha)$ could be infinite, for instance, when $\alpha$ is induced by a Dirac measure at $1.$ As $\Phi_{\alpha,\sqrt{\gamma}}(0,0)$ is strictly convex in $\alpha$ \cite{AChen14}, $\hg$ is convex on $\mathcal{M}$ and this convexity is strict along any linear path connecting two distinct $\alpha$ and $\alpha'$ with finite $\hg(\alpha)$ and $\hg(\alpha').$ Denote by $\mathcal{M}_\xi$ the collection of all Parisi measures associated to any $\beta>0$. Our next main result below says that the functionals $\hp(\gamma)$ and $\hg(\alpha)$ are related through a Legendre variational principle and $\alpha$ and $\gamma$ are conjugate variables.
\begin{theorem}[Legendre structure]\label{thm1}
	We have that
	\begin{align}
	\begin{split}
	\label{thm1:eq1}
	\hp(\gamma)&=\inf_{\alpha\in\mathcal{M}}\Bigl(\hg(\alpha)+\frac{\gamma}{2}\int_0^1\alpha(s)\xi'(s)ds\Bigr),
	\end{split}
	\end{align}
	where the infimum is uniquely attained by the Parisi measure $\alpha_{P,\sqrt{\gamma}}$. Conversely, if $\alpha\in\mathcal{M}_\xi$, then
	\begin{align}
	\begin{split}
	\label{thm1:eq2}
	\hg(\alpha)&=\sup_{\gamma\in(0,\infty)}\Bigl(\hp(\gamma)-\frac{\gamma}{2}\int_0^1\alpha(s)\xi'(s)ds\Bigr),
	\end{split}
	\end{align}
	where the supremum is reached by some $\gamma$ with $\alpha_{P,\sqrt{\gamma}}=\alpha.$
\end{theorem}

Equations \eqref{thm1:eq1} and \eqref{thm1:eq2} together describes a good Legendre structure as the second terms on the right-hand sides are both bilinear. From a numerical simulation, it seems to indicate that both Theorems \ref{thm2} and \ref{thm1} still hold if one adds an external field of the form, $\sqrt{\gamma} h$, for some $h\in\mathbb{R}$ to the model, but they will be invalid if the external field is simply $h$.

We add a few remarks here. First the maximization problem in \eqref{thm1:eq2} could have infinitely many maximizers. For instance, consider the SK model $\xi(t)=t^2/2$ and let $\alpha=1$ on $[0,1]$. It is known in \cite{ALR,G95} that $\alpha_{P,\beta}=1$ on $[0,1]$ if $\beta\in (0,1].$ This and \eqref{rmk2:eq1} below imply that the set of maximizers contains the interval $(0,1].$ Next note that the Parisi variational formula $\mathcal{P}(\beta)$ is a convex function, which can be deduced from the equality of \eqref{eq:ParisiEq23} and the convexity of the free energy, but it is by no means clear how to justify this property directly from the variation formula itself. In this regard, the new expression \eqref{thm1:eq1} for the reparametrized Parisi formula seems to be more natural as it not only retains the uniqueness of the minimizer,  but also straightforwardly leads to the concavity of $\hp(\gamma)$. Alternatively, one may derive the Legendre duality by considering the Legendre transform of $\hp(\gamma)$ rather than using $\hp(\alpha,\gamma)$, but one will then lose the uniqueness of the minimizer, see Remark \ref{rmk1} after the proof of Theorem \ref{thm1}. Now we give two direct consequences of Theorem \ref{thm1}.

\begin{proposition}\label{prop1}
	Let $0<\beta_1<\beta_2$ and let $\alpha_{P,\beta_1},\alpha_{P,\beta_2}$ be the associated Parisi measures for $\mathcal{P}(\beta_1),\mathcal{P}(\beta_2)$, respectively. Then we have
	\begin{align*}
	\lim_{N\rightarrow\infty}\e \left<\xi(R_{1,2})\right>_{\beta_1}=\int_0^1\xi d\alpha_{P,\beta_1}\leq \int_0^1\xi d\alpha_{P,\beta_2}=\lim_{N\rightarrow\infty}\e \left<\xi(R_{1,2})\right>_{\beta_2},
	\end{align*}
	where $\left<\cdot\right>_{\beta_1}$ and $\left<\cdot\right>_{\beta_2}$ are the Gibbs averages corresponding to the Boltzmann weights $e^{-\beta_1H_N}$ and $e^{-\beta_2H_N}$ and $R_{1,2}$ is the overlap between $\sigma^1$ and $\sigma^2$, two i.i.d. samplings from these Gibbs averages. 
\end{proposition}

Heuristically, one would expect some kind of stochastic monotonicity of the Parisi measures in the inverse temperature. Proposition \ref{prop1} gives a partial manifestation in this direction. Next, we show that the mixed $p$-spin model does not have a first order phase transition.

\begin{proposition}
	\label{prop0} 
	The mapping $\beta\mapsto \alpha_{P,\beta}$ is continuous and $\mathcal{P}(\beta)$ is continuously differentiable.
\end{proposition}


The rest of the paper is organized as follows. In next section, we will discuss the analogue of Theorems 1 and 2 for the REM discovered by Guerra \cite{G15}. In Section 3, we compute the derivative of the Parisi PDE solution  in temperature. The main novelty and crucial tool of the section is the use of the variational representation of the Parisi PDE solution obtained in authors' previous work \cite{AChen14}. In Section 4, we provide some auxiliary technical  lemmas while in Section 5, we build on the computations of Sections 3 and 4 to provide the proof of Theorems 1, 2 and Proposition \ref{prop1}. 

\smallskip
\smallskip

\noindent {\bf Acknowledgements.} W.-K. C. is indebted to Francesco Guerra for motivating this research work and conducting several enlightening discussions while he was participating in the trimester event ``Disordered Systems, Random Spatial Processes and Some Applications'' at the institute of Henri Poincar\'e in March 2015. A. A. and W.-K. C. thank Nicola Kistler for many fruitful discussions and bringing \cite{EN08,EN09} to their attention and Dmitry Panchenko for valuable suggestions regarding the presentation of the paper. The research of A. A. is supported by NSF grant DMS-1517864. The research of W.-K. C. is supported by NSF grant DMS-1513605, NSF-Simons Travel Grant, NSF Travel Support for IHP Trimester in Probability and Hong Kong Research Grants Council GRF-14302515.

\section{The Legendre formulation in the REM}

Recall that the REM is defined on the hypercube $\Sigma_N$ and its Hamiltonian $(X_N(\sigma):\sigma\in\Sigma_N)$ is a collection of i.i.d. standard Gaussian random variables. In this section, we will explain the Legendre structure in the REM model following Guerra \cite{G15}. Note that the thermodynamic limit of the free energy in the REM has a Parisi-type variational representation (see \cite{G15}),
\begin{align*}
\lim_{N\rightarrow\infty}\frac{1}{N}\e\log \sum_\sigma \exp(-\beta \sqrt{N}X_N(\sigma))=\mathcal{P}_{\mbox{\tiny REM}}(\beta),
\end{align*}
where for a given standard Gaussian random variable $z$, 
\begin{align*}
\mathcal{P}_{\mbox{\tiny REM}}(\beta):=\inf_{m\in[0,1]}\Bigl(\frac{\log 2}{m}+\frac{1}{m}\log \e e^{m\beta z}\Bigr)=\inf_{m\in[0,1]}\Bigl(\frac{\log 2}{m}+\frac{\beta^2 m}{2}\Bigr).
\end{align*}
From the last equation, if one considers the reparametrized Parisi formula, 
\begin{align*}
\hat{\mathcal{P}}_{\mbox{\tiny REM}}(\gamma):=\mathcal{P}_{\mbox{\tiny REM}}(\beta)
\end{align*}
for $\gamma=\beta^2,$ then $\hat{\mathcal{P}}_{\mbox{\tiny REM}}$ is a concave function of $\gamma.$ In fact, there is an interpolation argument of obtaining such concavity without knowing the explicit form of the Parisi-type formula, which is due to Guerra \cite{G15} and the argument runs essentially in the same way as \cite{G03}:

\begin{proposition} $\hat{\mathcal{P}}_{\mbox{\tiny REM}}(\gamma)$ is concave in $\gamma.$ 
\end{proposition}

\begin{proof}
	Let $\gamma_1,\gamma_2\geq0$ and $\gamma=(1-\lambda)\gamma_1+\lambda\gamma_2.$ For $M,N\geq 1$, write $\Sigma_{M+N}=\Sigma_M\times\Sigma_N.$ Let $(X_{M+N}(\tau):\tau\in\Sigma_{M+N})$, $(X_M(\rho):\rho\in\Sigma_M)$ and $(X_N(\sigma):\sigma\in\Sigma_N)$ be i.i.d. standard Gaussian random variables. Consider partition functions,
	\begin{align*}
	Z_{M+N}({\gamma})&=\sum_{\tau}e^{-\sqrt{\gamma(M+N)}X_{M+N}(\tau)},\\
	Z_M({\gamma_1})&=\sum_{\rho}e^{-\sqrt{\gamma_1M}X_M(\rho)},\\
	Z_N({\gamma_2})&=\sum_{\sigma}e^{-\sqrt{\gamma_2N}X_N(\sigma)}.
	\end{align*}
	Set the interpolated Hamiltonian,
	$$
	X_{N,t}(\tau)=\sqrt{t\gamma(M+N)}X_{M+N}(\tau)+\sqrt{1-t}(\sqrt{\gamma_1M}X_M(\rho)+\sqrt{\gamma_2N}X_N(\sigma))
	$$
	for $\tau=(\rho,\sigma)\in\Sigma_{M+N}$ and define
	\begin{align*}
	\varphi(t)&=\frac{1}{M+N}\e\log \sum_{{\tau}}\exp X_{N,t}(\tau), 
	\end{align*}
	where the summation is over all $\tau=(\rho,\sigma)\in\Sigma_{M+N}.$ In particular, 
	\begin{align*}
	\varphi(0)&=\frac{M}{M+N}\bigg(\frac{1}{M} \mathbb E \log Z_M({ \gamma_{1}})\bigg)+\frac{N}{M+N}\bigg(\frac{1}{N} \mathbb E \log Z_N({ \gamma_{2}})\bigg),\\
	\varphi(1)&=\frac{1}{M+N} \mathbb E \log Z_{M+N}(\gamma).
	\end{align*}
	Now from Gaussian integration by parts formula, one may compute 
	\begin{align}\label{prop2:proof:eq1}
	\varphi'(t)&= -\frac{1}{2}\e\Bigl<\Bigl(\gamma \delta_{\rho^1,\rho^2}\delta_{\sigma^1,\sigma^2}-\frac{\gamma_1M}{M+N}\delta_{\rho^1,\rho^2}-\frac{\gamma_2N}{M+N}\delta_{\sigma^1,\sigma^2}\Bigr)\Bigr>_t,
	\end{align}
	where $\left<\cdot\right>_t$ is the Gibbs expectation with respect to Hamiltonian $X_{N,t}$ and $\delta_{a,b}$ is the Dirac measure. If now we take $N/(M+N)\rightarrow\lambda,$ then the huge term inside the Gibbs expectation will be non-positive when $M,N$ tends to infinity. Consequently, 
	\begin{align*}
	\hp_{\tiny REM}({\gamma})=\lim_{N+M\rightarrow\infty}\varphi(1)&\geq\lim_{N,M\rightarrow\infty}\varphi(0)=(1-\lambda) \hp_{\tiny REM}({\gamma_1})+\lambda\hp_{\tiny REM}({\gamma_2}).
	\end{align*}
\end{proof}

Remark that the proceeding approach does not work in establishing the concavity of the reparametrized Parisi formula in the mixed $p$-spin model as now one sees no clear way why in the limit the derivative of the corresponding interpolation should have a nonnegative sign in this case. The Legendre structure of $\hp_{\mbox{\tiny REM}}$ is established as follows. Consider the Legendre transform of $\hp_{\mbox{\tiny REM}}(\gamma)$ by
$$
\hat{\Gamma}_{\tiny \mbox{REM}}(m)=\sup_{\gamma\geq 0}\Bigl(\hat{\mathcal{P}}_{\mbox{\tiny REM}}(\gamma)-\frac{\gamma m}{2}\Bigr).
$$
Note that a straightforward computation gives \begin{align*}
\hat{\mathcal{P}}_{\mbox{\tiny REM}}(\gamma)&=
\left\{
\begin{array}{ll}
\frac{\gamma}{2}+\log 2,&\mbox{if $\gamma\leq {2\log 2}$},\\
\\
\sqrt{2\gamma\log 2},&\mbox{if $\gamma>{2\log 2}$},
\end{array}
\right.
\end{align*}
which implies that for $0<m\leq 1,$ 
\begin{align*}
\Bigl(\hat{\mathcal{P}}_{\mbox{\tiny REM}}(\gamma)-\frac{\gamma m}{2}\Bigr)'&=
\left\{
\begin{array}{ll}
\frac{1}{2}(1-m),&\mbox{if $\gamma\leq {2\log 2}$},\\
\\
\frac{1}{2}\sqrt{\frac{2\log 2}{\gamma}}-\frac{m}{2},&\mbox{if $\gamma>{2\log 2}$}.
\end{array}
\right.
\end{align*}
As a result, one sees that the maximizer of the variational problem $\hg_{\mbox{\tiny REM}}$ is any $\gamma\leq 2\log 2$ if $m=1$ and $\gamma=2\log 2/m^2$ if $0<m<1.$ Consequently, $\hat{\Gamma}_{\tiny \mbox{REM}}(m)=\log 2/m$ and thus, from the Parisi-type formula, the Legendre conjugacy holds
\begin{align*}
\hat{\mathcal{P}}_{\mbox{\tiny REM}}(\gamma)&=\inf_{m\in[0,1]}\Bigl(\hat{\Gamma}_{ \mbox{\tiny REM}}(m)+\frac{\gamma m}{2}\Bigr),
\end{align*}
where as in the case of the mixed $p$-spin model, $\hat{\Gamma}_{\tiny \mbox{REM}}(m)$ is a convex function and the second term is both linear in $m$ and $\gamma.$

\section{Derivative of the Parisi PDE solution in temperature}

For the rest of the paper, we denote $$\Psi_{\alpha,\gamma}=\Phi_{\alpha,\sqrt{\gamma}}\,\,\mbox{and}\,\,\zeta=\xi''.$$ 
Our first step to prove the concavity of $\Psi_{\alpha,\gamma}$ is to get a nice expression for the derivative of this function in $\gamma$. For technical purpose and simplicity, we shall take advantage of the variational representation for the Parisi PDE solution established in \cite{AChen14}, which we describe now. Let $W$ be a standard Brownian motion. For any $\alpha\in\mathcal{M}$ and $\gamma\geq 0,$ recall from \cite[Theorem 2]{AChen14} that the Parisi PDE solution $\Psi_{\alpha,\gamma}$ admits the following variational representation,
\begin{align}\label{lem1:proof:eq0}
\Psi_{\alpha,\gamma}(0,x)=\max_{u}\e\Bigl[\log \cosh\Bigl(x+\gamma\int_0^1\alpha\zeta uds+\gamma^{1/2}\int_0^1\zeta^{1/2}dW\Bigr)-\frac{\gamma}{2}\int_0^1\alpha\zeta u^2ds\Bigr],
\end{align}
where $u$ is taken over all progressively measurable processes with respect to the filtration generated by $W$. Here the maximum is realized by the process $$u_{\alpha,\gamma}^x(s):=\partial_x\Psi_{\alpha,\gamma}(s,X_{\alpha,\gamma}(s)),\,\,\forall s\in[0,1]$$ for $X_{\alpha,\gamma}=(X_{\alpha,\gamma}(s))_{0\leq s\leq 1}$ satisfying
\begin{align*}
dX_{\alpha,\gamma}&=\gamma\alpha\zeta\partial_x\Psi_{\alpha,\gamma}(s,X_{\alpha,\gamma})ds+\gamma^{1/2}\zeta^{1/2}dW
\end{align*}
with initial condition $X_{\alpha,\gamma}(0)=x.$ Originally introduced in \cite{AChen14}, the variational representation \eqref {lem1:proof:eq0} was used to establish the strict convexity of the Parisi PDE solution in $\alpha$. The argument therein relies on some further properties about the maximizer and we found that those are equally useful in our computation. More precisely, from \cite[Lemma 2]{AChen14}, one has two identities, for $0\leq t\leq t'\leq 1,$
\begin{align}\label{lem1:proof:eq1}
\partial_x\Psi_{\alpha,\gamma}(t',X_{\alpha,\gamma}(t'))-\partial_x\Psi_{\alpha,\gamma}(t,X_{\alpha,\gamma}(t))&=\gamma^{1/2}\int_t^{t'} \zeta^{1/2}\partial_{xx}\Psi_{\alpha,\gamma}(s,X_{\alpha,\gamma})dW
\end{align} 
and
\begin{align}
\begin{split}\label{lem1:proof:eq2}
&\partial_{xx}\Psi_{\alpha,\gamma}(t',X_{\alpha,\gamma}(t'))-\partial_{xx}\Psi_{\alpha,\gamma}(t,X_{\alpha,\gamma}(t))\\
&=-\gamma\int_t^{t'}\alpha \zeta(\partial_{xx}\Psi_{\alpha,\gamma}(s,X_{\alpha,\gamma}))^2ds+\gamma^{1/2}\int_t^{t'}\zeta^{1/2}\partial_{x}^3\Psi_{\alpha,\gamma}(s,X_{\alpha,\gamma})dW.
\end{split}
\end{align}

\begin{proposition}\label{lem1}
	For any $\gamma>0$, we have that
	\begin{align}\label{lem1:eq1}
	\partial_\gamma \Psi_{\alpha,\gamma}(0,x)&=\frac{1}{2}\Bigl(\xi'(1)-\int_0^1\xi'\e (u_{\alpha,\gamma}^x)^2d\alpha\Bigr).
	\end{align}
\end{proposition}

\begin{proof} 
	Since one may approximate the Parisi PDE solution $\Psi_{\alpha,\gamma}$ as well as the process $X_{\alpha,\gamma}$ by using the distribution functions induced by atomic probability measures (see the proof of \cite[Theorem 2]{AChen14}), we consider without loss of generality that $\alpha$ satisfies 
	\begin{align}
	\label{eq1}
	\mbox{$\alpha(s)=m_\ell$ for $s\in [q_\ell,q_{\ell+1})$ for $0\leq \ell \leq k$ and $\alpha(1)=1$,}
	\end{align} 
	where
	\begin{align*}
	0&=q_0\leq q_1\leq \cdots\leq q_{k}\leq q_{k+1}=1,\\
	0&=m_0\leq m_1\leq \cdots \leq m_{k-1}\leq m_{k}=1.
	\end{align*} 
	From \cite[Lemma 2]{C15}, we compute by using \eqref{lem1:proof:eq0} and then \eqref{lem1:proof:eq1} to get
	\begin{align}
	\begin{split}
	\notag
	\partial_\gamma \Psi_{\alpha,\gamma}(0,x)&=\e\Bigl[ u_{\alpha,\gamma}^x(1)\Bigl(\int_0^1\alpha\zeta u_{\alpha,\gamma}^xds+\frac{1}{2\sqrt{\gamma}}\int_0^1\zeta^{1/2}dW\Bigr)\Bigr]-\frac{1}{2}\int_0^1\alpha \zeta\e(u_{\alpha,\gamma}^x)^2ds
	\end{split}\\
	\begin{split}
	\label{lem1:proof:eq3}
	&=\frac{1}{2}\int_0^1\alpha \zeta\e (u_{\alpha,\gamma}^x)^2ds+\frac{1}{2\sqrt{\gamma}}\e u_{\alpha,\gamma}^x(1)\int_0^1\zeta^{1/2}dW.
	\end{split}
	\end{align}
	Using integration by parts and \eqref{lem1:proof:eq1}, the first term can be computed as
	\begin{align*}
	\int_0^1 \alpha\zeta\e (u_{\alpha,\gamma}^x)^2ds
	&=\sum_{\ell=0}^{k}m_\ell \int_{q_\ell}^{q_{\ell+1}}\zeta\e (u_{\alpha,\gamma}^x)^2ds\\
	&=\sum_{\ell=0}^{k}m_\ell\Bigl(\xi'\e( u_{\alpha,\gamma}^x)^2\big|_{q_\ell}^{q_{\ell+1}}-\int_{q_\ell}^{q_{\ell+1}}\xi'\frac{d}{ds}\e (u_{\alpha,\gamma}^x)^2ds\Bigr)\\
	&=\sum_{\ell=0}^{k}m_\ell \xi'\e (u_{\alpha,\gamma}^x)^2\big|_{q_\ell}^{q_{\ell+1}}-\gamma\int_0^1\alpha\xi'\zeta\e \partial_{xx}\Psi_{\alpha,\gamma}(s,X_{\alpha,\gamma})^2ds.
	\end{align*}
	As for the second term, we use \eqref{lem1:proof:eq1}, integration by parts, and then \eqref{lem1:proof:eq2},
	\begin{align*}
	\e u_{\alpha,\gamma}^x(1)\int_0^1\zeta^{1/2}dW
	&=\gamma^{1/2}\e \int_0^1 \zeta^{1/2}\partial_{xx}\Psi_{\alpha,\gamma}(s,X_{\alpha,\gamma})dW\cdot \int_0^1\zeta^{1/2}dW\\
	&=\gamma^{1/2}\e \int_0^1\zeta\e \partial_{xx}\Psi_{\alpha,\gamma}(s,X_{\alpha,\gamma})ds\\
	&=\gamma^{1/2}\Bigl(\xi'\e \partial_{xx}\Psi_{\alpha,\gamma}(s,X_{\alpha,\gamma})\big|_0^1-\int_0^1\xi'\frac{d}{ds}\e \partial_{xx}\Psi_{\alpha,\gamma}(s,X_{\alpha,\gamma})ds\Bigr)\\
	&=\gamma^{1/2}\Bigl(\xi'(1)(1-\e u_{\alpha,\gamma}^x(1)^2)+\gamma\int_0^1\alpha\xi'\zeta\e \partial_{xx}\Psi_{\alpha,\gamma}(s,X_{\alpha,\gamma})^2ds\Bigr),
	\end{align*}
	where the last equality used $\partial_{xx}\Psi_{\alpha,\gamma}(1,y)=1-\tanh^2(y)=1-\partial_x\Psi_{\alpha,\gamma}(1,y)^2$ and $\xi'(0)=0.$
	Putting these into \eqref{lem1:proof:eq3} yields
	\begin{align*}
	\partial_\gamma \Psi_{\alpha,\gamma}(0,x)&=\frac{1}{2}\Bigl(\sum_{\ell=0}^{k}m_\ell \xi'\e (u_{\alpha,\gamma}^x)^2\big|_{q_\ell}^{q_{\ell+1}}+\xi'(1)(1-\e u_{\alpha,\gamma}^x(1)^2)\Bigr),
	\end{align*}
	which implies \eqref{lem1:eq1} since
	\begin{align*}
	\sum_{\ell=0}^{k}m_\ell \xi'\e (u_{\alpha,\gamma}^x)^2\big|_{q_\ell}^{q_{\ell+1}}
	&=\sum_{\ell=0}^{k-1} m_\ell \xi'(q_{\ell+1})\e u_{\alpha,\gamma}^x(q_{\ell+1})^2-\sum_{\ell=0}^{k-1} m_{\ell+1} \xi'(q_{\ell+1})\e u_{\alpha,\gamma}^x(q_{\ell+1})^2\\
	&\quad +\xi'(1)\e u_{\alpha,\gamma}^x(1)^2-\xi'(0)\cdot \e u_{\alpha,\gamma}^x(0)^2\\
	&=-\int_0^1 \xi'\e (u_{\alpha,\gamma}^x)^2d\alpha+ \xi'(1)\e u_{\alpha,\gamma}^x(1)^2,
	\end{align*}
	where again we used $\xi'(0)=0$ in the second equality.
	
\end{proof}

\begin{remark}\rm
Applying \eqref{lem1:eq1} with $x=0$ and the identity
$$
\int_0^1\alpha(s) s\xi''(s)ds=\xi'(1)-\int_0^1\alpha(s)\xi'(s)ds-\int_0^1 s\xi'(s)d\alpha,
$$ 
one obtains
	\begin{align}\label{cor2:eq1}
	\partial_\gamma \hp({\alpha,\gamma})&=\frac{1}{2}\Bigl(\int_0^1\alpha(s)\xi'(s)ds-\int_0^1\xi'(s)(\e (u_{\alpha,\gamma}^0(s))^2-s)d\alpha\Bigr).
	\end{align}
In particular, if $\alpha$ is the Parisi measure $\alpha_{P,\sqrt{\gamma}}$, then from \cite[Proposition 1]{C15}, $$\e u_{\alpha_{P,\sqrt{\gamma}},\gamma}^0(s)^2=s$$ for any $s$ in the support of $\alpha$ and thus, \eqref{cor2:eq1} and \cite[Lemma 2]{C15} together yields
\begin{align}\label{rmk2:eq1}
\hp'(\gamma)=\partial_\gamma\hp(\alpha_{P,\sqrt{\gamma}},\gamma)&=\frac{1}{2}\int_0^1\alpha_{P,\sqrt{\gamma}}(s)\xi'(s)ds.
\end{align}
Note that a similar equation for $\mathcal{P}$ was also derived in Panchenko \cite{PD08}, where heavily using the property of $\alpha_{P,\beta}$ being the minimizer of the Parisi functional, he presented an elementary argument to obtain
\begin{align}\label{diff}
\mathcal{P}'(\beta)&=\beta\int_0^1\alpha_{P,\beta}(s)\xi'(s)ds.
\end{align}
In our situation, \eqref{cor2:eq1} holds for arbitrary $\alpha$ and the derivation is more delicate using the variational formula for the Parisi PDE. Another possible approach to justifying Proposition \ref{lem1} could be a direct computation via the Gaussian integration by parts formula. However, this leads to extensive computations and it seems unclear to the authors how to simplify the final expression into a simple formula as in Proposition \ref{lem1}.  
\end{remark}

Now by the virtue of \eqref{lem1:eq1}, it is clear that the concavity of $\Psi_{\alpha,\gamma}$ will follow if one could establish that $\e u_{\alpha,\gamma}^0(s)^2$ is nondecreasing in $\gamma$ for any $s\in[0,1].$ To this end, we shall perform a change of variables to express $\e u_{\alpha,\gamma}^0(s)^2$ in terms of Gaussian random variables as follows. Let $\alpha$ satisfy \eqref{eq1}. Set independent Gaussian random variables
\begin{align}
\begin{split}
\label{def1}
z_0,z_1,\ldots,z_k
\end{split}
\end{align} 
with mean zero and variance  $\e z_j^2=\xi'(q_{j+1})-\xi'(q_j)$ and define for any $0\leq a\leq b\leq k+1$, 
\begin{align}\label{def2}
\zeta_{a,b}=\sum_{j=a}^{b-1}z_j,\,\,\eta_{a,b}^x=\sqrt{\gamma}(x+\zeta_{a,b}).
\end{align}
Here, the case $a=b$ should be understood as $\zeta_{a,a}=0$ and $\eta_{a,a}^x=\sqrt{\gamma}x.$
\begin{proposition}\label{prop2}
	Assume that $\alpha$ satisfies \eqref{eq1}. Let $f$ be a bounded measurable function on $\mathbb{R}.$ 
	For $0\leq a\leq b\leq k+1$, we have that
	\begin{align}
	\begin{split}
	\label{lem10}
	&\e f(X_{a,b}^x(q_b))
	=\e f(\eta_{a,b}^x)\exp \sum_{\ell=a}^{b-1}m_\ell(\Psi_{\alpha,\gamma}(q_{\ell+1},\eta_{a,{\ell+1}}^x)-\Psi_{\alpha,\gamma}(q_\ell,\eta_{a,\ell}^x)),
	\end{split}
	\end{align}	
	where $X^x_{a,b}(r)$ for $q_a\leq r\leq q_{b}$ is the solution to the following SDE,
	$$
	dX_{a,b}^x=\gamma\alpha \zeta\Psi_{\alpha,\gamma}(r,X_{a,b}^x)dr+\gamma^{1/2}\zeta^{1/2}dW
	$$
	with $X_{a,b}^x(q_a)=x$.
\end{proposition}

\begin{proof}
	Obviously if $a=b$, the announced equation holds. Assume that \eqref{lem10} is true for some $1\leq a\leq b.$ Let $u$ be an arbitrary progressively measurable process on $[q_a,q_b]$. Set
	$$
	Z=\exp\Bigl(-\frac{\gamma}{2}\int_{q_{a-1}}^{q_a}m_{a-1}^2\zeta u^2dr-\gamma^{1/2}\int_{q_{a-1}}^{q_{a}}m_{a-1}\zeta^{1/2}udW\Bigr).
	$$
	Define $d\tilde{\mathbb{P}}=Zd\mathbb{P}$ and $$\tilde{W}(s)=\gamma^{1/2}\int_{q_{a-1}}^{s}m_{a-1}\zeta^{1/2}udr+W(s).$$ Then the Girsanov theorem says that $\tilde{W}$ is a Brownian motion under $\tilde{\mathbb P}$. Denote
	\begin{align*}
	Y(t)&=x+\gamma^{1/2}\int_{q_{a-1}}^{t}\zeta^{1/2}dW.
	\end{align*}
	Let $(\tilde{Y}(t))_{q_{a-1}\leq t\leq q_a}$ be the solution to 
	\begin{align*}
	\tilde{Y}(t)&=x+\gamma\int_{q_{a-1}}^{t}m_{a-1}\zeta \partial_x\Psi_{\alpha,\eta}(r,\tilde{Y})dr+\gamma^{1/2}\int_{q_{a-1}}^{t}\zeta^{1/2}dW
	\end{align*}
	with $\tilde{Y}(q_{a-1})=x.$
	We now take $u(r)=\partial_x\Psi_{\alpha,\gamma}(r,\tilde{Y})$ and express
	\begin{align*}
	\tilde{Y}(t)&=x+\gamma^{1/2}\int_{q_{a-1}}^t\zeta^{1/2}d\tilde{W}.
	\end{align*}
	Then the induction hypothesis gives
	\begin{align}
	\begin{split}\notag
	&\e f(\eta_{a-1,b}^x)\exp \sum_{\ell=a-1}^{b-1}m_\ell(\Psi_{\alpha,\gamma}\bigl(q_{\ell+1},\eta_{a-1,{\ell+1}}^x)-\Psi_{\alpha,\gamma}(q_\ell,\eta_{a-1,\ell}^x)\bigr)\\ 
	&=\e f\bigl(X_{a,b}^{Y(q_a)}(q_b)\bigr)\exp m_\ell(\Psi_{\alpha,\gamma}\bigl(q_{a},Y(q_a))-\Psi_{\alpha,\gamma}(q_{a-1},x)\bigr)\\
	&=\tilde\e f\bigl(X_{a,b}^{\tilde Y(q_a)}(q_b)\bigr)\exp m_\ell(\Psi_{\alpha,\gamma}\bigl(q_{a},{\tilde Y}(q_a))-\Psi_{\alpha,\gamma}(q_{a-1},x)\bigr)\\
	&=\e f\bigl(X_{a,b}^{\tilde Y(q_a)}(q_b)\bigr)\exp m_\ell(\Psi_{\alpha,\gamma}\bigl(q_{a},{\tilde Y}(q_a))-\Psi_{\alpha,\gamma}(q_{a-1},x)\bigr)
	\end{split}\\
	\begin{split}\label{lem10:proof:eq1}
	&\quad\cdot \exp\Bigl(-\frac{\gamma}{2}\int_{q_{a-1}}^{q_a}m_{a-1}^2\zeta u^2dr-\gamma^{1/2}\int_{q_{a-1}}^{q_{a}}m_{a-1}\zeta^{1/2}udW\Bigr).
	\end{split}
	\end{align}
	Here clearly $X_{a,b}^{\tilde{Y}(q_{a})}(q_b)=X_{a-1,b}^x(q_b).$ On the other hand, from It\^{o}'s formula and \eqref{pde},
	\begin{align*}
	&\Psi_{\alpha,\gamma}\bigl(q_{a},{\tilde Y}(q_a))-\Psi_{\alpha,\gamma}(q_{a-1},x)\\
	&=\int_{q_{a-1}}^{q_a}\partial_t\Psi_{\alpha,\gamma}(t,\tilde{Y})dt+\int_{q_{a-1}}^{q_a}\partial_x\Psi_{\alpha,\gamma}(t,\tilde{Y})d\tilde{Y}+\frac{1}{2}\int_{q_{a-1}}^{q_a}
	\partial_{xx}\Psi_{\alpha,\gamma}(t,\tilde{Y})d\left<\tilde{Y}\right>\\
	&=\int_{q_{a-1}}^{q_a}\partial_t\Psi_{\alpha,\gamma}(t,\tilde{Y})dt+\int_{q_{a-1}}^{q_a}u(\gamma m_{a-1}\zeta udr+\gamma^{1/2}\zeta^{1/2}dW)
	+\frac{\gamma}{2}\int_{q_{a-1}}^{q_a}\zeta
	\partial_{xx}\Psi_{\alpha,\gamma}(t,\tilde{Y})dt\\
	&=\frac{\gamma m_{a-1}}{2}\int_{q_{a-1}}^{q_a}\zeta u^2dr+\gamma^{1/2}\int_{q_{a-1}}^{q_a}\zeta^{1/2}udW,
	\end{align*}
	where the last equality used \eqref{pde}. These and \eqref{lem10:proof:eq1} conclude the announced result for the case when $a$ is replaced by $a-1$.
\end{proof}

Letting $f(y):=\partial_x\Psi_{\alpha,\gamma}(q_b,y)^2$ and $a=0,$ Proposition \ref{prop2} reads

\begin{corollary}\label{cor1}
	For $0\leq b\leq k+1$,
	\begin{align*}
	&\e u_{\alpha,\gamma}^x(q_b)^2
	=\e \partial_x\Psi_{\alpha,\gamma}(q_b,\eta_{0,b}^x)^2\exp \sum_{\ell=0}^{b-1}m_\ell(\Psi_{\alpha,\gamma}(q_{\ell+1},\eta_{0,{\ell+1}}^x)-\Psi_{\alpha,\gamma}(q_\ell,\eta_{0,\ell}^x)).
	\end{align*}		
\end{corollary}

\section{Some Auxiliary Lemmas}
This section is a preparation for establishing the monotonicity of $\e u_{\alpha,\gamma}^0(s)^2$ in $\gamma$ for any $s\in[0,1],$ which is consisted of three lemmas. The first gathers a few properties about the expectations for functions of Gaussian random variables as well as two covariance inequalities, one is a special case of the FKG inequality, while the other is taken from \cite{P08}. 

\begin{lemma}\label{lem2} 
	Suppose that $f,f_1,f_2$ are real-valued functions on $\mathbb{R}$ and $z$ is a centered Gaussian random varaible with $\e z^2=c^2.$ 
	\begin{itemize}
		\item[$(i)$] If $f_1,f_2$ are odd with $f_1\leq f_2$ on $[0,\infty)$ then $\e f_1(x+z)\leq \e f_2(x+z)$ for all $x\geq 0.$
		\item[$(ii)$] Let $D$ be a nonnegative function on $\mathbb{R}$ with $\e D(x+z)=1$ for some $x\geq 0.$ 		
		If either
	\begin{align}
	\label{lem2:cond0}
	\mbox{$f_1,f_2$ are even and nondecreasing on $[0,\infty)$}
	\end{align}
	or
		\begin{align}
		\label{lem2:cond1}
		\left\{
		\begin{array}{l}
		\mbox{$f_1$ is even and $f_2$ is odd},\\
		\mbox{$f_1,f_2$ are nondecreasing on $[0,\infty)$},\\
		\mbox{$D$ is even,}
		\end{array}
		\right.
		\end{align}
		then we have
		\begin{align}
		\begin{split}\label{lem2:eq1}
		&\e f_1(x+z)f_2(x+z)D(x+z)\geq \e f_1(x+z)D(x+z)\e f_2(x+z)D(x+z).
		\end{split}
		\end{align}
	
	\end{itemize}
\end{lemma}

\begin{proof} Note that 
	\begin{align*}
	\e f_l(x+z)&=\int_{-\infty}^\infty f_l(u)\rho(u,x)\exp\bigl(\frac{ux}{c^2}\bigr)du,
	\end{align*}
	where $\rho(u,x)=(2\pi c^2)^{-1/2}\exp(-(u^2+x^2)/2c^2).$ If we first split this integral into two parts $[0,\infty)$ and $(-\infty,0]$ and then using change of variables $v=-u$ and the assumption that $f_l$ is odd, it follows that  
	\begin{align*}
	\e f_l(x+z)&=\int_{0}^\infty f_l(u)\rho(u,x)\exp\bigl(\frac{ux}{c^2}\bigr)du-\int_{0}^{\infty} f_l(v)\rho(v,x)\exp\bigl(-\frac{vx}{c^2}\bigr)dv\\
	&=2\int_0^\infty f_l(u)\rho(u,x)\sinh\bigl(\frac{ux}{c^2}\bigr)du.
	\end{align*}
	Since $\sinh(ux)\geq 0$ for $x,u\geq 0$ and $f_1\leq f_2,$ this equation gives $(i).$
	
	\smallskip
	
	As for $(ii),$ let $z'$ be an independent copy of $z.$ Denote $z_x=x+z$ and $z_x'=x+z'.$ Using $\e D(z_x)=\e D(z_x')=1,$ we write
	\begin{align*}
	&\e f_1(z_x)f_2(z_x)D(z_x)
	-\e f_1(z_x)D(z_x)\e f_2(z_x)D(z_x)\\
	&=\e D(z_x)D(z_x')(f_1(z_x)-f_1(z_x'))(f_2(z_x)-f_2(z_x'))I(z\geq z').
	\end{align*}
	Applying change of variables $(s,t)=(z_x,z_x'),$ this integral equals
	\begin{align}\label{lem2:proof:eq2}
	\int_{\{s\geq t\}} K(s,t)\exp\biggl(-\frac{1}{2c^2}((s-x)^2+(t-x)^2)\biggr)dsdt,
	\end{align}
	where 
	$$
	K(s,t):=\frac{1}{2\pi c^2} D(s)D(t)(f_1(s)-f_1(t))(f_2(s)-f_2(t)).
	$$
	Assume that \eqref{lem2:cond0} holds. Clearly,
	$$
	(f_1(s)-f_1(t))(f_2(s)-f_2(t))=(f_1(|s|)-f_1(|t|))(f_2(|s|)-f_2(|t|))\geq0
	$$
	for all $s,t\in\mathbb{R}$ since $f_1,f_2$ are even and are nondecreasing on $[0,\infty).$ Thus, \eqref{lem2:eq1} holds. To prove \eqref{lem2:eq1} under the assumption \eqref{lem2:cond1}, let us split the integral region of \eqref{lem2:proof:eq2} into two parts $\Omega_1=\{(s,t):s\geq t,|s|\geq |t|\}$ and $\Omega_2=\{(s,t):s\geq t,|s|< |t|\}.$ Using change of variables $(u,v)=(-t,-s)$ and the assumptions that $f_1$ is even, $f_2$ is odd and $D$ is even, we obtain
	\begin{align*}
	&\int_{\Omega_2} K(s,t)\exp\biggl(-\frac{1}{2c^2}((s-x)^2+(t-x)^2)\biggr)dsdt\\
	&=-\int_{\Omega_1} K(u,v)\exp\biggl(-\frac{1}{2c^2}((u+x)^2+(v+x)^2)\biggr)dudv
	\end{align*}
	and thus, \eqref{lem2:proof:eq2} becomes $\int_{\Omega_1}K(s,t)L(x,s,t)dsdt,$
	where 
	\begin{align*}
	L(x,s,t)&:=\exp\biggl(-\frac{1}{2c^2}((s-x)^2+(t-x)^2)\biggr)-\exp\biggl(-\frac{1}{2c^2}((s+x)^2+(t+x)^2)\biggr).
	\end{align*}
	Note that from \eqref{lem2:cond1}, it implies $K\geq 0$ on $\Omega_1.$ Also since $x\geq 0$ and $s+t\geq 0$ on $\Omega_1,$ this gives
	$
	(s+x)^2+(t+x)^2-(s-x)^2-(t-x)^2= 4x(s+t)\geq 0
	$
	and thus, $L\geq 0$ on $\Omega_1$. These imply \eqref{lem2:eq1}.
\end{proof}

The second one is a comparison lemma.

\begin{lemma}\label{lem3}
	Let $0\leq \gamma_1\leq \gamma_2$ and $m>0.$ Suppose that $A_1,A_2$ are even convex with $$A_1'(\sqrt{\gamma_1}x)\leq A_2'(\sqrt{\gamma_2}x),\,\,\forall x\geq 0.$$ Define 
	\begin{align*}
	B_1(x)&:=\frac{1}{m}\log \e \exp mA_1(x+\sqrt{\gamma_1}z),\\
	B_2(x)&:=\frac{1}{m}\log \e\exp mA_2(x+\sqrt{\gamma_2}z)
	\end{align*}
	for some centered Gaussian random variable $z$. 
	Then
	\begin{align}
	\begin{split}\label{lem3:eq2}
	B_1'(\sqrt{\gamma_1}x)&\leq B_2'(\sqrt{\gamma_2}x),\,\,\forall x\geq 0.
	\end{split}
	\end{align} 
	Moreover, if $C_1,C_2$ are even with $$C_1(\sqrt{\gamma_1}x)\leq C_2(\sqrt{\gamma_2}x),\,\,\forall x\geq 0$$
	and $
	C_1',C_2'\geq 0$ on $[0,\infty),$
	then
	\begin{align}
	\begin{split}\label{lem3:eq1}
	&\e C_1(\sqrt{\gamma_1}(x+z))\exp m\bigl(A_1(\sqrt{\gamma_1}(x+z))-B_1(\sqrt{\gamma_1}x)\bigr)\\
	&\leq \e C_2(\sqrt{\gamma_2}(x+z))\exp m\bigl(A_2(\sqrt{\gamma_2}(x+z))-B_2(\sqrt{\gamma_2}x)\bigr),\,\,\forall x\in\mathbb{R}.
	\end{split}
	\end{align}

\end{lemma}

\begin{proof} Denote $z_1=\sqrt{\gamma_1}(x+z)$ and $z_2=\sqrt{\gamma_2}(x+z).$ For \eqref{lem3:eq2}, we consider
	$$
	{\rho}(\lambda):=\frac{\e A_1'(z_1)S_\lambda}{\e {S}_\lambda},\,\,\forall \lambda\in[0,1],
	$$
	where $S_\lambda=\exp({m((1-\lambda)A_1(z_1)+\lambda A_2(z_2))}).$
	Then
	\begin{align*}
	\rho'(\lambda)&=m\frac{\e A_1'(z_1)(A_2(z_2)-A_1(z_1))S_\lambda\cdot \e S_\lambda- \e A_1'(z_1)S_\lambda \cdot\e (A_2(z_2)-A_1(z_1))S_\lambda}{(\e S_\lambda)^2}.
	\end{align*}
	Let $f_1(y)=A_2(\sqrt{\gamma_2}y)-A_1(\sqrt{\gamma_1}y)$ and $f_2(y)=A_1'(\sqrt{\gamma_1}y)$. Using the even convexity of $A_1,A_2$  and $A_2'(\sqrt{\gamma_2}y)\geq A_1'(\sqrt{\gamma_1}y)$ for $y\geq 0,$ one sees that $f_1$ is even and
	\begin{align*}
	f_1'(y)&=\sqrt{\gamma_2}A_2'(\sqrt{\gamma_2}y)-\sqrt{\gamma_1}A_1'(\sqrt{\gamma_1}y)\\
	&=(\sqrt{\gamma_2}-\sqrt{\gamma_1})A_2'(\sqrt{\gamma_2}y)+\sqrt{\gamma_1}(A_2'(\sqrt{\gamma_2}y)-A_1'(\sqrt{\gamma_1}y))\geq 0,\,\,\forall y\geq 0,
	\end{align*}
	and that $f_2$ is odd with nonnegative derivative on $[0,\infty).$ In addition, note that $$
	D_\lambda(y):=\frac{\exp m((1-\lambda)A_1(\sqrt{\gamma_1}y)+\lambda A_2(\sqrt{\gamma_2}y))}{\e S_\lambda}
	$$ is even and $\e D_\lambda(x+z)=1.$ Consequently, plugging $f_1, f_2$ and $D_\lambda$ into \eqref{lem2:cond1} and \eqref{lem2:eq1} leads to
	${\rho}'\geq 0.$ So
	\begin{align}
	\begin{split}\notag
	B_1'(\sqrt{\gamma_1}x)&=\e A_1'(z_1)\exp m\bigl(A_1(z_1)-B_1(\sqrt{\gamma_1}x)\bigr)\\
	&={\rho}(0)\\
	&\leq \rho(1)
	\end{split}\\
	\begin{split}
	\label{lem2:proof:eq1}
	&=\e A_1'(z_1)\exp m\bigl(A_2(z_2)-B_2(\sqrt{\gamma_2}x)\bigr).
	\end{split}
	\end{align} 
	Finally, since $A_1'(\sqrt{\gamma_1}y)D_1(y)\leq A_2'(\sqrt{\gamma_2}y)D_1(y)$ for all $y\geq 0$ and the two sides of this inequality are odd functions, we use Lemma \ref{lem2} $(i)$ to get
	\begin{align*}
	\e A_1'(z_1)\exp m\bigl(A_2(z_2)-B_2(\sqrt{\gamma_2}x)\bigr)
	&\leq \e A_2'(z_2)\exp m\bigl(A_2(z_2)-B_2(\sqrt{\gamma_2}x)\bigr)\\
	&=B_2'(\sqrt{\gamma_2}x).
	\end{align*}
	This inequality and \eqref{lem2:proof:eq1} together gives \eqref{lem3:eq2}. As for \eqref{lem3:eq1}, since $C_1,C_2$ are even, it suffices to prove \eqref{lem3:eq1} only for $x\geq 0.$ Define 
	$$
	\hat\rho(\lambda)=\frac{\e C_1(z_1)S_\lambda}{\e S_\lambda},\,\,\forall \lambda\in[0,1].
	$$
	Computing directly gives
	\begin{align*}
	\hat\rho'(\lambda)&=m\frac{\e C_1(z_1)(A_2(z_2)-A_1(z_1))S_\lambda\cdot \e S_\lambda- \e C_1(z_1)S_\lambda \cdot\e (A_2(z_2)-A_1(z_1))S_\lambda}{(\e S_\lambda)^2}
	\end{align*}
	Set $f_3(y)=C_1(\sqrt{\gamma_1}y)$. Note that $f_1,f_3$ are both even and have nonnegative derivatives on $[0,\infty).$ These allow us to apply \eqref{lem2:cond0}  and \eqref{lem2:eq1} to obtain $\hat\rho'\geq 0$ for $\lambda\in[0,1].$ As a result, \eqref{lem3:eq1} follows from
	\begin{align*}
	\e C_1(z_1)\exp m\bigl(A_1(z_1)-B_1(\sqrt{\gamma_1}x)\bigr)
	&=\hat\rho(0)\\
	&\leq \hat\rho(1)\\
	&=	\e C_1(z_1)\exp m\bigl(A_2(z_2)-B_2(\sqrt{\gamma_2}x)\bigr)\\
	&\leq \e C_2(z_2)\exp m\bigl(A_2(z_2)-B_2(\sqrt{\gamma_2}x)\bigr),
	\end{align*}
	where the last inequality used the assumptions that $C_1,C_2$ are even and $C_1(\sqrt{\gamma_1}y)\leq C_2(\sqrt{\gamma_2}y)$ for $y\geq 0.$
	This completes our proof.
	
\end{proof}

\begin{lemma}
	\label{lem8}
	Let $A$ be even convex and $C$ be even with $C'\geq 0$ on $[0,\infty).$ Then 
	\begin{align*}
	F(x):=\frac{\e C(x+\sqrt{\gamma}z)\exp mA(x+\sqrt{\gamma}z)}{\e \exp mA(x+\sqrt{\gamma}z)}
	\end{align*}
	is even and is nondecreasing for $x\geq 0.$
\end{lemma}

\begin{proof}
	Easy to see that $F$ is even. Denote $z_x=x+\sqrt{\gamma}z$. We compute directly to get
	\begin{align*}
	F'(x)&=\frac{\e C'(z_x)e^{mA(z_x)}}{\e e^{mA(z_x)}}\\
	&+m\frac{\e C(z_x)A'(z_x)e^{mA(z_x)}\cdot \e e^{mA(z_x)}-\e C(z_x)e^{mA(z_x)}\cdot \e A'(z_x)e^{mA(z_x)}}{(\e e^{mA(z_x)})^2}
	\end{align*}
	Here since $C'e^{mA}$ is odd and $\geq 0$ on $[0,\infty)$, Lemma \ref{lem2} $(i)$ implies that the first term is nonnegative. As for the second term, since $C$ is even with $C'\geq 0$ on $[0,\infty)$, $A'$ is odd with $A''\geq 0$ on $[0,\infty)$ and $e^{mA}$ is even, the application of \eqref{lem2:cond1} and \eqref{lem2:eq1} implies that it is also nonnegative. So $F$ is nondecreasing for $x\geq 0.$
\end{proof}

\section{Proof of Main Results}

We will first prove that $\Psi_{\alpha,\gamma}(s)$ is concave in $\gamma$ for any $s\in[0,1]$ and then establish the Legendre structure of the Parisi formula. 
Before we start, note that from \cite[Proposition 2]{AChen}, $\Psi_{\alpha,\gamma}$ is a twice differentiable even convex function in the spacial variable, which will be used over and over again in our argument. Assume that $\alpha$ is of the form \eqref{eq1}. Recall $z_0,\ldots,z_{k+1}$ from \eqref{def1} and $\zeta_{a,b},\eta_{a,b}^x$ from \eqref{def2}. Let $0\leq \gamma_1\leq\gamma_2.$ We set $$\eta_{a,b}^{i,x}=\gamma_i^{1/2}(x+\zeta_{a,b})$$ for $0\leq a\leq b\leq k+1$ and $i=1,2.$ 

\begin{lemma}
	\label{prop3}
	For $0\leq b\leq k+1,$ $\partial_x\Psi_{\alpha,\gamma_1}(q_b,\sqrt{\gamma_1}x)\leq \partial_x\Psi_{\alpha,\gamma_2}(q_b,\sqrt{\gamma_2}x)$ for $x\geq 0.$
\end{lemma}

\begin{proof}
	We proceed by induction in $b.$ Since $\Psi_{\alpha,\gamma_1}(1,x)=\log\cosh(x)=\Psi_{\alpha,\gamma_2}(1,x)$ and $(\log\cosh x)'=\tanh x$ is increasing, Lemma \ref{prop3} follows with $b=k+1$.
	Assume that \eqref{prop3} is true for some $1\leq b\leq k+1.$ Letting $A_i(x):=\Psi_{\alpha,\gamma_i}(q_b,x)$, we have $$
	B_i(x):=\frac{1}{m_{b-1}}\log \e \exp m_{b-1}A_i(x+\sqrt{\gamma_i}z_{b-1})=\Psi_{\alpha,\gamma_i}(q_{b-1},x).
	$$
	It is known that $A_i$ is even and convex. From the induction hypothesis, we also get $B_1'(\sqrt{\gamma_1}x)\leq B_2'(\sqrt{\gamma_2}x).$ 
	Thus, \eqref{lem3:eq2} concludes our statement in the case that $b$ is replaced by $b-1$ and we are done.
\end{proof}

\begin{lemma}\label{lem9}
	For any $0\leq a\leq b\leq k+1,$ the following is nondecreasing in $x\geq 0,$
	\begin{align}
	\begin{split}\label{lem9:eq1}
	&\e \partial_x\Psi_{\alpha,\gamma_i}(q_b,x+\sqrt{\gamma_i}\zeta_{a,b})^2\\
	&\qquad\cdot\exp \sum_{\ell=a}^{b-1}m_\ell(\Psi_{\alpha,\gamma_i}\bigl(q_{\ell+1},x+\sqrt{\gamma}\zeta_{a,\ell+1})-\Psi_{\alpha,\gamma_i}(q_\ell,x+\sqrt{\gamma_i}\zeta_{a,\ell})\bigr).
	\end{split}
	\end{align}
\end{lemma}

\begin{proof}
	We argue by induction in $a.$
	If $a=b,$ then the announced inequality reads
	\begin{align*}
	\partial_x\Psi_{\alpha,\gamma_i}(q_b,x)^2\leq\partial_x\Psi_{\alpha,\gamma_i}(q_b,y)^2
	\end{align*}
	for any $0\leq x\leq y$ since we have known that $\Psi_{\alpha,\gamma_i}(q_b,\cdot)$ is even and convex.  
	Assume that \eqref{lem9:eq1} holds for some $1\leq a\leq b$. Define $A=\Psi_{\alpha,\gamma_i}(q_a,\cdot)$ and denote \eqref{lem9:eq1} by $C(x)$.  Now we express
	\begin{align*}
	&\e \partial_x\Psi_{\alpha,\gamma_i}(q_b,x+\sqrt{\gamma}\zeta_{a-1,b})^2\\
	&\qquad\cdot\exp \sum_{\ell={a-1}}^{b-1}m_\ell(\Psi_{\alpha,\gamma_i}\bigl(q_{\ell+1},x+\sqrt{\gamma}\zeta_{a-1,\ell+1})-\Psi_{\alpha,\gamma_i}(q_\ell,x+\sqrt{\gamma}\zeta_{a-1,\ell})\bigr)\\
	&=\frac{\e C(x+\sqrt{\gamma_i}z_{a-1})\exp m_{a-1}A(x+\sqrt{\gamma_i}z_{a-1})}{\e \exp m_{a-1}A(x+\sqrt{\gamma_i}z_{a-1})},
	\end{align*}
	where we used
	\begin{align*}
	\Psi_{\alpha,\gamma_i}(q_{a-1},x)=\frac{1}{m_{a-1}}\log \e\exp m_{a-1}\Psi_{\alpha,\gamma_i}(q_a,x+\sqrt{\gamma_i}z_{a-1}).
	\end{align*}
	Since $A$ is even convex and $C$ is even with $C'\geq 0$ on $[0,\infty)$ from the induction hypothesis, Lemma \ref{lem8} shows that \eqref{lem9:eq1} is valid with $a$ replaced by $a-1$. This completes our proof.
	
\end{proof}

\begin{proof}[\bf Proof of Theorem \ref{thm2}] From Proposition \ref{lem1}, we only need to show that 
	\begin{align}\label{thm2:proof:eq1}
	\e u_{\alpha,\gamma_1}^0(q)^2\leq \e u_{\alpha,\gamma_2}^0(q)^2,\,\,\forall q\in[0,1]\,\,\mbox{and}\,\,0\leq \gamma_1\leq \gamma_2.
	\end{align} 
	By an approximation argument, it is sufficient to consider $\alpha$'s of the form \eqref{eq1} and establish \eqref{thm2:proof:eq1} for $q=q_0,\ldots,q_{k+1}.$
	To this end, we claim that for $0\leq a\leq b\leq k+1$,
	\begin{align}
	\begin{split}\label{lem7:eq1}
	&\e \partial_x\Psi_{\alpha,\gamma_1}(q_b,\eta_{a,b}^{1,x})^2\exp \sum_{\ell=a}^{b-1}m_\ell\bigl(\Psi_{\alpha,\gamma_1}(q_{\ell+1},\eta_{a,\ell+1}^{1,x})-\Psi_{\alpha,\gamma_1}(q_\ell,\eta_{a,\ell}^{1,x})\bigr)\\
	&\leq \e \partial_x\Psi_{\alpha,\gamma_2}(q_b,\eta_{a,b}^{2,x})^2\exp \sum_{\ell=a}^{b-1}m_\ell\bigl(\Psi_{\alpha,\gamma_2}(q_{\ell+1},\eta_{a,\ell+1}^{2,x})-\Psi_{\alpha,\gamma_2}(q_\ell,\eta_{a,\ell}^{2,x})\bigr),\,\,\forall x\in\mathbb{R}.
	\end{split}
	\end{align}
	If this holds, taking $x=0$ and $a=0$ and applying Corollary \ref{cor1} to this inequality gives \eqref{thm2:proof:eq1} for $q=q_b$ with $0\leq b\leq k+1$ and therefore ends our proof. To justify \eqref{lem7:eq1}, we again argue by induction on $a$.
	Note that from Lemma \ref{prop3},
	\begin{align*}
	\partial_x\Psi_{\alpha,\gamma_1}(q_b,\sqrt{\gamma_1}x)^2\leq \partial_x\Psi_{\alpha,\gamma_2}(q_b,\sqrt{\gamma_2}x)^2,\,\,\forall x\in\mathbb{R}.
	\end{align*}
	This gives the base case $a=b$ of \eqref{lem7:eq1}. Assume that \eqref{lem7:eq1} holds for some $1\leq a\leq b.$
	Set 
	\begin{align*}
	A_i(y)&=\Psi_{\alpha,\gamma_i}(q_{a},y),\\
	B_i(y)&=\frac{1}{m_{a-1}}\log \e\exp m_{a-1}A_i(y+\sqrt{\gamma_i}z_{a-1})=\Psi_{\alpha,\gamma_i}(q_{a-1},y)
	\end{align*}
	and
	\begin{align*}
	C_i(y)&=\e \partial_x\Psi_{\alpha,\gamma_i}(q_a,x+\sqrt{\gamma_i}\zeta_{a,b})^2\\
	&\qquad\cdot\exp \sum_{\ell=a}^{b-1}m_\ell\bigl(\Psi_{\alpha,\gamma_i}(q_{\ell+1},y+\sqrt{\gamma_i}\zeta_{a,\ell+1})-\Psi_{\alpha,\gamma_i}(q_\ell,y+\sqrt{\gamma_i}\zeta_{a,\ell})\bigr).
	\end{align*}
	Note that $A_1,A_2$ are even convex with $A_1'(\sqrt{\gamma_1}y)\leq A_2'(\sqrt{\gamma_2}y)$ for $y\geq 0$ by Lemma \ref{prop3} and that $C_1,C_2$ are even and nondecreasing for $y\geq 0$ by Lemma \ref{lem9}. 
	Applying \eqref{lem3:eq1}, we obtain
	\begin{align*}
	&\e C_1(\sqrt{\gamma_1}(x+z_{a-1}))\exp m_{a-1}\bigl(A_1(\sqrt{\gamma_1}(x+z_{a-1}))-B_1(\sqrt{\gamma_1}x)\bigr)\\
	&\leq \e C_2(\sqrt{\gamma_2}(x+z_{a-1}))\exp m_{a-1}\bigl(A_2(\sqrt{\gamma_2}(x+z_{a-1}))-B_2(\sqrt{\gamma_2}x)\bigr),
	\end{align*}
	which gives \eqref{lem7:eq1} in the case that $a$ is replaced by $a-1$ since
	\begin{align*}
	&\e C_i(\sqrt{\gamma_i}(x+z_{a-1}))\exp m_{a-1}\bigl(A_i(\sqrt{\gamma_i}(x+z_{a-1}))-B_i(\sqrt{\gamma_i}x)\bigr)\\
	&=\e \partial_x\Psi_{\alpha,\gamma_i}(q_b,\eta_{a-1,b}^{i,x})^2\exp \sum_{\ell=a-1}^{b-1}m_\ell\bigl(\Psi_{\alpha,\gamma_i}(q_{\ell+1},\eta_{a-1,\ell+1}^{i,x})-\Psi_{\alpha,\gamma_i}(q_\ell,\eta_{a-1,\ell}^{i,x})\bigr).
	\end{align*}
	So our claim follows.
\end{proof}

\begin{proof}[\bf Proof of Theorem \ref{thm1}]
	It is clear from the definitions of $\hg$ that \begin{align*}
	\inf_{\alpha\in\mathcal{M}}\Bigl(\hg(\alpha)+\frac{\gamma}{2}\int_0^1\alpha(s)\xi'(s)ds\Bigr)\geq \hp(\gamma).
	\end{align*}
	Assume that $\alpha_{P,\sqrt{\gamma}}$ is the minimizer of the problem $\hp(\gamma)$. Use of \eqref{rmk2:eq1} implies
	\begin{align*}
	&\partial_{\gamma}\Bigl(\hp(\alpha_{P,\sqrt{\gamma}},\gamma)-\frac{\gamma}{2}\int_0^1\alpha_{P,\sqrt{\gamma}}(s)\xi'(s)ds\Bigr)\\
	&=\frac{1}{2}\int_0^1\alpha_{P,\sqrt{\gamma}}(s)\xi'(s)ds-\frac{1}{2}\int_0^1\alpha_{P,\sqrt{\gamma}}(s)\xi'(s)ds\\
	&=0.
	\end{align*}
	Since  $\hp(\alpha_{P,\sqrt{\gamma}},\cdot)$ is concave from Theorem \ref{thm2}, $\gamma$ is a maximizer for the variational problem $\hg(\alpha_{P,\sqrt{\gamma}})$. As a result,
	\begin{align*}
	&\hg(\alpha_{P,\sqrt{\gamma}})+\frac{\gamma}{2}\int_0^1\alpha_{P,\sqrt{\gamma}}(s)\xi'(s)ds\\
	&=\hp(\alpha_{P,\sqrt{\gamma}},\gamma)-\frac{\gamma}{2}\int_0^1\alpha_{P,\sqrt{\gamma}}(s)\xi'(s)ds+
	\frac{\gamma}{2}\int_0^1\alpha_{P,\sqrt{\gamma}}(s)\xi'(s)ds\\
	&= \hp(\gamma)
	\end{align*}
	and thus,
	\begin{align*}
	\inf_{\alpha\in\mathcal{M}}\Bigl(\hg(\alpha)+\frac{\gamma}{2}\int_0^1\alpha(s)\xi'(s)ds\Bigr)\leq \hp(\gamma).
	\end{align*}
	This gives \eqref{thm1:eq1} and the infimum for \eqref{thm1:eq1} is attained by the Parisi measure $\alpha_{P,\sqrt{\gamma}},$ while the uniqueness can be concluded from the fact, derived from the strict convexity of $\Phi_{\alpha,\sqrt{\gamma}}(0,0)$ in $\alpha$ \cite{AChen14}, that $\hg$ is strictly convex along any linear path joining two distinct $\alpha$ and $\alpha'$ with finite $\hg(\alpha)$ and $\hg(\alpha')$. As for \eqref{thm1:eq2}, clearly the left-hand side is no less than the right-hand side. The other direction could be obtained by a similar argument as above and using the crucial assumption that $\alpha$ is now a Parisi measure. We omit this part of the argument.
\end{proof} 

\begin{remark}\label{rmk1}\rm
	Another Legendre duality $\hp(\gamma)$ one could also have is to consider the Legendre transform of $\hp(\gamma)$ instead of $\hp(\alpha,\gamma),$
	\begin{align*}
	\hat{L}(\alpha)=\sup_{\gamma\geq 0}\Bigl(\hp(\gamma)-\frac{\gamma}{2}\int_0^1 \alpha(s)\xi'(s)ds\Bigr).
	\end{align*} 
	Note that it is known from \cite[Equation $(43)$]{AChen}, 
	$$
	\int_0^1\alpha_{P,\sqrt{\gamma}}(s)\xi'(s)ds\leq \sqrt{\frac{2\xi(1)\log 2}{\gamma}},
	$$
	which implies from \eqref{rmk2:eq1},
	\begin{align}
	\begin{split}
	\label{rmk1:eq3}
	\Bigl(\hp(\gamma)-\frac{\gamma}{2}\int_0^1 \alpha(s)\xi'(s)ds\Bigr)'&=\frac{1}{2}\Bigl(\int_0^1\alpha_{P,\sqrt{\gamma}}(s)\xi'(s)ds-\int_0^1\alpha(s)\xi'(s)ds\Bigr)
	\end{split}\\
	\begin{split}
	\label{rmk1:eq4}
	&\leq \frac{1}{2}\Bigl(\sqrt{\frac{2\xi(1)\log 2}{\gamma}}-\int_0^1\alpha(s)\xi'(s)ds\Bigr).
	\end{split}
	\end{align}
	From \eqref{rmk1:eq3}, if $\alpha(s)=0$ on $[0,1)$, then this derivative is positive for all $\gamma>0$ and thus $\hat{L}(\alpha)=\infty;$ otherwise the derivative is eventually negative when $\gamma$ is large enough from \eqref{rmk1:eq4}, in which case $\hat{L}(\alpha)<\infty.$ Now following exactly the same argument as Theorem \ref{thm1} concludes
	\begin{align}\label{rmk1:eq1}
	\hp(\gamma)&=\inf_{\alpha\in\mathcal{M}}\Bigl(\hat{L}(\alpha)+\frac{\gamma}{2}\int_0^1 \alpha(s)\xi'(s)ds\Bigr).
	\end{align}
	It is clear that $\hat{L}$ is convex by definition, but it is not strict and the minimizer in \eqref{rmk1:eq1} is not unique. Indeed, for any $\gamma>0,$ assume that $\alpha_0\in\mathcal{M}$ satisfies
	\begin{align}
	\label{rmk1:eq2}
	\int_0^1\alpha_0(s)\xi'(s)ds=\int_0^1\alpha_{P,\sqrt{\gamma}}(s)\xi'(s)ds.
	\end{align}
	From this and \eqref{rmk2:eq1}, we have
	\begin{align*}
	\Bigl(\hp(\gamma)-\frac{\gamma}{2}\int_0^1 \alpha_0(s)\xi'(s)ds\Bigr)'&=\frac{1}{2}\Bigl(\int_0^1\alpha_{P,\sqrt{\gamma}}(s)\xi'(s)ds-\int_0^1\alpha_0(s)\xi'(s)ds\Bigr)=0.
	\end{align*}
	Thus, by the concavity of $\hp,$
	\begin{align*}
	\hat{L}(\alpha_0)&=\hp(\gamma)-\frac{\gamma}{2}\int_0^1\alpha_0(s)\xi'(s)ds,
	\end{align*}
	which allows us to conclude two facts. First, $\hat{L}$ is not strict convex since any convex combination of two probability distributions satisfying \eqref{rmk1:eq2} also fulfills \eqref{rmk1:eq2}. 
	Second, $\alpha_0$ is a minimizer of \eqref{rmk1:eq1} since
	\begin{align*}
	\hp(\gamma)&\leq \hat{L}(\alpha_0)+\frac{\gamma}{2}\int_0^1\alpha_0(s)\xi'(s)ds=\hp(\gamma).
	\end{align*}
	So as long as $\alpha_{P,\sqrt{\gamma}}$ is not induced by a Dirac measure at $\{0\}$, or equivalently, $\alpha_{P,\sqrt{\gamma}}$ is not identically equal to $1$ on $[0,1]$, one can easily construct infinitely many minimizers through \eqref{rmk1:eq2}. 
\end{remark}

\begin{proof}[\bf Proof of Proposition \ref{prop1}]
	Let $\gamma_1=\beta_1^2$ and $\gamma_2=\beta_2^2.$
	From the concavity of $\hp(\gamma)$ and the formula \eqref{rmk2:eq1}, one sees that
	\begin{align*}
	\frac{1}{2}\int_0^1\alpha_{P,\sqrt{\gamma_1}}(s)\xi'(s)ds&=\hp'(\gamma_1)\geq\hp'(\gamma_2)=\frac{1}{2}\int_0^1\alpha_{P,\sqrt{\gamma_2}}(s)\xi'(s)ds,
	\end{align*} 	
	which from integration by parts implies that $$
	\int_0^1\xi(s)d\alpha_{P,\sqrt{\gamma_1}}\leq \int_0^1\xi(s)d\alpha_{P,\sqrt{\gamma_2}}.
	$$
	Replacing $\sqrt{\gamma_1}$ and $\sqrt{\gamma_2}$ respectively by $\beta_1$ and $\beta_2$ and using the fact that $$\lim_{N\rightarrow\infty}\e\left<\xi(R_{1,2})\right>_{\beta_i}=\int_0^1\xi(s)d\alpha_{P,\beta_i}$$ from \cite{PD08} complete our proof.
\end{proof}

\begin{proof}[\bf Proof of Proposition \ref{prop0}]
	Note that $\mathcal{P}(\beta)$ is continuous for all $\beta>0$ and from Theorem \ref{thm1}, it can be written as
	\begin{align*}
	\mathcal{P}(\beta)&=\hg(\alpha_{P,\beta})+\frac{\beta^2}{2}\int_0^1\alpha_{P,\beta}(s)\xi'(s)ds.
	\end{align*}
	To prove the continuity of $\alpha_{P,\beta}$ in $\beta>0,$ let $(\beta_n)$ be any positive sequence with limit $\beta.$ From compactness of $\mathcal{M}$, it suffices to assume that $\alpha_{P,\beta_n}$ converges to some $\alpha_0\in\mathcal{M}.$ Thus,
	\begin{align*}
	\mathcal{P}(\beta)
	&=\lim_{n\rightarrow\infty}\mathcal{P}(\beta_n)\\
	&=\lim_{n\rightarrow\infty}\bigl(\hg(\alpha_{P,\beta_n})+\frac{\beta_n^2}{2}\int_0^1\alpha_{P,\beta_n}(s)\xi'(s)ds\bigr)\\
	&=\hg(\alpha_{0})+\frac{\beta^2}{2}\int_0^1\alpha_{0}(s)\xi'(s)ds,
	\end{align*}	
	from which we conclude $\alpha_0=\alpha_{P,\beta}$ by the uniquensss of the minimizer in the problem \eqref{thm1:eq1} and this gives the continuity of $\beta\mapsto \alpha_{P,\beta}$. Finally, 
	using \eqref{diff} completes our proof.
\end{proof}


\end{document}